\documentclass[12pt,a4paper,reqno]{amsart}

\usepackage{amsmath,amsthm,amsfonts,amssymb,euscript,cite, mathrsfs,mathtools}
\usepackage{graphics}
\usepackage{xcolor} %
\usepackage{hyperref}
\hypersetup{colorlinks,linkcolor={blue},citecolor={blue},urlcolor={red}}  
\usepackage{enumitem} %
\usepackage{ulem}
\usepackage[noabbrev,capitalize,nameinlink]{cleveref}
\usepackage{setspace}
\usepackage{enumitem}

\crefname{equation}{}{}
\crefname{algocf}{Algorithm}{Algorithms}
\crefname{equation}{}{} %
\crefname{algocf}{Algorithm}{Algorithms}

\usepackage{graphicx}
\usepackage{color}
\usepackage{transparent}

\usepackage{fullpage} %

\usepackage{microtype} %

\usepackage{seqsplit}

\definecolor{green}{rgb}{0,0.8,0} %

\definecolor{babypink}{rgb}{0.96,0.76,0.76}

\newcommand{\rt}{{\mathbb R^3}}

\newcommand{\pa}{\partial}

\newcommand{\grad}{\nabla_x}

\newcommand{\gab}{g^{\alpha\beta}}
\newcommand{\hab}{h^{\alpha\beta}}
\newcommand{\Hab}{H^{\x\xb}}

\newcommand{\pab}{\partial_\beta}
\newcommand{\paa}{\partial_\alpha}
\newcommand{\paab}{\partial_{\alpha\beta}}

\newcommand{\pat}{\partial_t}

\newcommand{\pai}{\partial_i}

\newcommand{\pao}{\partial\mkern-10mu /\,}

\newcommand{\ti}{\tilde}

\newcommand{\la}{\langle}
\newcommand{\ra}{\rangle}

\newcommand{\ls}{\lesssim}

\newcommand{\inv}{^{-1}}
\newcommand{\invh}{^{-\f12}}

\newcommand{\de}{\nabla_{t,x}} %

\newcommand{\f}{\frac}

\newcommand{\iy}{\infty}
\renewcommand{\S}{{\mathbb S}}

\newcommand{\crt}{{C^{ R}_{ T}}}
\newcommand{\cut}{{C^{ U}_{ T}}}

\newcommand{\crtt}{{\ti C^{ R}_{ T}}}

\newcommand{\licut}{{L^\iy\lr{C^{ U}_{ T}}}}

\newcommand{\inte}{{C^{<3T/4}_{ T}}}

\newcommand{\co}{{D_{tr}^{R}}}
\newcommand{\dtr}{{D_{tr}}}

\newcommand{\lolt}{{L^1L^{2}}}
\newcommand{\ltlt}{{L^2L^2}}

\newcommand{\lt}{L^2}
\newcommand{\p}{\phi}

\newcommand{\supp}{\text{supp\,}}

\let\pminus\pm
\renewcommand{\pm}{\phi_{\le m}}

\newcommand{\pmn}{\p_{\le m+n}}
\newcommand{\nm}{\la t-r\ra} %
\newcommand{\jr}{\la r\ra} %
\newcommand{\ju}{\la u\ra}

\newcommand{\jt}{\la t\ra}
\newcommand{\jv}{\la v\ra}

\newcommand{\lr}[1]{\left( #1 \right)}

\newcommand{\tpsi}{{\widetilde{\psi}}}

\newcommand{\leo}{{LE^1}}

\def\doi#1{ {\href{http://dx.doi.org/#1}
   {{\mdseries\ttfamily DOI}}}}

\newtheorem{theorem}{Theorem}[section]
\newtheorem{corollary}[theorem]{Corollary}
\newtheorem{lemma}[theorem]{Lemma}
\newtheorem{proposition}[theorem]{Proposition}
\theoremstyle{definition}
\crefname{claim}{Claim}{Claims}

\newtheorem{definition}[theorem]{Definition}

\theoremstyle{remark}
\newtheorem{remark}[theorem]{Remark}
\theoremstyle{conjecture}

\numberwithin{equation}{section}

\newcommand{\x}{\alpha}
\newcommand{\xb}{\beta}
\newcommand{\xd}{\delta}

\newcommand{\eps}{\epsilon}

\newcommand{\xk}{\kappa}

\newcommand{\xo}{\omega}
	\newcommand{\xO}{\Omega}

	\newcommand{\xS}{\Sigma}

\newcommand{\N}{{\mathbb N}}

\newcommand{\R}{\mathbb R}

\newcommand{\calM}{\mathcal M}

\newcommand{\calR}{\mathcal R}

\newcommand{\calT}{\mathcal T}

\newcommand{\calV}{\mathcal V}

\newcommand{\Lbar}{\underline{L}}
\newcommand{\uL}{\Lbar}

\begin{document}

\title{Improved decay for quasilinear wave equations close to asymptotically flat spacetimes including black hole spacetimes}
\author{Shi-Zhuo Looi}
\address{Department of Mathematics, University of Kentucky, Lexington, KY 40506}
\email{Shizhuo.Looi@gmail.com}

\begin{abstract}We study the quasilinear wave equation $\Box_{g}\phi=0$ where the metric $g = g(\p,t,x)$ is close to and asymptotically approaches $g(0,t,x)$, which equals the Schwarzschild metric or a Kerr metric with small angular momentum, as time tends to infinity. 
	Under only weak assumptions on the metric coefficients, we prove an improved pointwise decay rate for the solution $\p$. One consequence of this rate is that for bounded $|x|$, we have the integrable decay rate $|\p(t,x)| \le Ct^{-1-\min(\xd,1)}$ where $\xd>0$ is a parameter governing the decay, near the light cone, of the coefficient of the slowest-decaying term in the quasilinearity.
		We obtain the same aforementioned pointwise decay rates for the quasilinear wave equation $(\Box_{\ti g} + B^{\x}(t,x)\paa + V(t,x))\p=0$ with a more general asymptotically flat metric $\ti g =\ti g(\p,t,x)$ and with other time-dependent asymptotically flat lower order terms. 
\end{abstract}

\maketitle

\section{Introduction}

We study the quasilinear initial value problem 	%
\begin{equation}   \label{eq:problem}
\begin{cases}
 \Box_{g_{B}}\phi(\ti t,x) =  f(\pa^2\p,\p,\ti t,x), \qquad (\ti t,x) \in \calM, \ \ f :=  (\Hab(\ti t,x) \p + O(\p^{2})) \paa\pab \p \\
\p(\ti t=0,x) = \p_0(x), \quad \ti T \p(\ti t=0,x) = \p_1(x)
\end{cases}.
\end{equation}where $g_{B}$ denotes either the Schwarzschild metric or the Kerr metric with small angular momentum, 
$$\Box_g = |g|^{-1/2} \pa_\mu(|g|^{1/2} g^{\mu\nu} \pa_\nu )$$ 
and $\calM:= \{ r>r_e\}$ (for some constant\footnote{In the Kerr case, $r_e$ is a number such that $r_e \in (r_-,r_+)$ where $r_->0$ corresponds to the Cauchy horizon and $r_+$ to the event horizon. See also the appendix (\cref{app}) for more detail on the Kerr metric. } $r_e>0$) is a manifold which contains the domain of outer communication. 
	We also study the problem with a more general asymptotically flat operator on the left-hand side of the equation in \cref{eq:problem}. 
	Here, $\ti t$ is a coordinate with the property that $\{ \ti t=0\}$ is spacelike and $\ti t = t$ away from the black hole, and 
	$\ti T$ denotes a smooth and everywhere-timelike vector field such that $\ti T =\pat$ away from the black hole. 
	To state a simpler version of our main theorem, we introduce the Schwarzschild spacetime with mass $M>0$: it is a spherically symmetric solution of the Einstein vacuum equation and is given by 
	$$g = -\lr{1 - \f{2M}r} d t^2+ \lr{1-\f{2M}r}\inv dr^2 + r^2 g_\xo$$
on $\R_{t} \times (2M,\iy)_r \times \mathbb S^2$ where $g_\xo$ is the standard metric on $\mathbb S^2$. Henceforth for simplicity of notation we shall write $t$ to denote $\ti t$. 

\begin{theorem}[Simple special case of main theorem] \label{thm:simple}
Let $g_B$ in \cref{eq:problem} denote the Schwarzschild metric and let $\p$ denote the global solution to the problem \cref{eq:problem} with small, smooth and compactly supported initial data $(\p_0,\p_1)$, and let $0 < r_e < 2M$, where $\{r=2M\}$ denotes the event horizon. 
Fix a compact subset $K$ of $(r_e,\iy)\times \mathbb S^2$. Then there exists a constant $C>0$ such that 
$$|\p(t,x)| \le C t^{-1-\min(\xd,1)}, \quad x \in K$$where $\xd>0$ is a number such that a certain subset of the coefficients $\Hab$ obeys the upper bound $(1 +|t-r|)^\xd / (1 + t)^\xd$ near the light cone $r=t$ (see \cref{met.ass}). 
Applying any number or combination of partial derivatives, rotation vector fields $\xO_{ij} = x^i \pai - x^j \pai$, or the scaling vector field $S = t\pat + r\pa_r$ to $\phi$ does not change the above estimate (but changes $C$). 

In addition, the same pointwise decay rate for $|\p(t,x)|$ holds for a more general class of asymptotically flat and time-dependent metrics, and it also holds if we add both first-order and zeroth-order terms that are asymptotically flat and time-dependent to the equation. 
\end{theorem}
	
	We are motivated to study the long-time behaviour of solutions to \cref{eq:problem} by way of the black hole stability problem which is, roughly speaking, the problem of showing that solutions to the Einstein equations with initial data close to a Kerr solution have domains of outer communication that converge toward a Kerr solution. In wave coordinates, the Einstein equations are a system of quasilinear wave equations, and a toy model of this system is given by equation \cref{eq:problem}. Thus \cref{eq:problem} is a toy model for the Einstein equations close to the Schwarzschild metric or a Kerr metric with small angular momentum $a$. See also \cref{app} (the appendix) for the definition of the Kerr metric. 

The articles \cite{LinToh1} and \cite{LinToh2} proved global existence for \cref{eq:problem} using a bootstrap for the Schwarzschild metric and Kerr metric with small $a$ respectively, under suitable assumptions on the metric coefficients for sufficiently small initial data. The bulk of their arguments was used to prove an integrated local energy decay (ILED) estimate\footnote{It is known as a Weak ILED estimate which is a weaker version of the usual ILED estimate. See \cref{def:EB} for the energy estimate and \cref{LED} for the usual ILED estimate.}, after which a bootstrap shows global existence in short order. Another immediate consequence of ILED is\footnote{See Theorem 5.3 in \cite{LinToh2}. The statement of Theorem 5.3 there is a $\jt\inv\la t-r^{*}\ra^{1/2}$ decay rate, but we simply refer to this as a $t^{-1/2}$ decay rate.} a $t^{-1/2}$ decay rate for the solution for bounded $|x|$. 

In this paper, we extend their results and contribute to the understanding of the long-time behaviour of \cref{eq:problem} by showing an integrable decay rate $|\p| \le Ct^{-1-\xd}$ if $\xd>0$ is a small parameter less than or equal to 1, and we moreover show that $|\p| \le Ct^{-2}$ if $\xd$ is taken to be larger than 1.
	Informally, the main theorem (\cref{thm:main}) states that assuming known energy bounds and ILED bounds for the problem \cref{eq:problem}, we obtain the aforementioned decay rate. 
	The parameter $\xd$ arises from the assumption \cref{met.ass} governing the decay rate of one of the $H$ coefficients of the quasilinearity $f$---more specifically, its decay rate near the light cone. 
		For comparison, we note that the linear problem $$\Box_{g_{B}}\p=0$$ exhibits a pointwise decay rate of $|\p(t,x)| \le Ct^{-3}$ for bounded $|x|$, %
		while as \cref{thm:simple,thm:main} indicate, our result implies that for 
		$$\Box_{g_B}\p = f(\partial^2\p,\p,t,x),$$
	we have $|\p(t,x)| \le Ct^{-1-\min(\xd,1)}$ for bounded $|x|$.

For $\Box_{g_B}\p=0$ with $g_B$ equal to the Schwarzschild metric, the solution to the wave equation was conjectured to decay at the rate of $t^{-3}$ on a compact region in \cite{Pri}. This rate of decay was shown to hold for the Schwarzschild spacetime, the subextremal Kerr spacetime with $|a|<M$, and other spacetimes; see, for instance, the works \cite{DSS,Tat,MTT,Hin,AAG3} and also the works \cite{AAG4,MaZhang1,MaZhang2}. 

\subsection*{Coordinates, and the operator $P$} %
We make a change of coordinates which near spatial infinity involves using precisely the Regge-Wheeler coordinates, in either Schwarzschild or Kerr. %
After a further conformal transformation, we can replace $\Box_{g}$ by an operator of the form 
\begin{equation}\label{P def}
P := \Box+g^{\xo}(t,r)\Delta_{\xo} + \paa s_{2}^{\x\xb}(t,x)\pab + s_{3}(t,x), \quad s_{j} \in S^{Z}(\jr^{-j}), \ g^{\xo}\in S^{Z}_{\text{radial}}(\jr^{-3})
\end{equation}where $\Box$ denotes the d'Alembertian in the Minkowski metric and $\jr := (1 + r^2)^{1/2}$. Here the functions $\{s_{2}^{\x\xb}\}_{\x,\xb}$ arise in the setting of Kerr, whereas they are equal to 0 in the setting of Schwarzschild.\footnote{The reader looking for more detail about normalised coordinates is encouraged to refer to Sections 2.1, 2.2 and 2.3 of \cite{Tat}, and here we only provide an overview of this change of coordinates.} %
We call these normalised coordinates.

Let $\ti r = \ti r(r)$ denote a smooth and strictly increasing function that equals $r$ for $r \le R_{1}$ and $r^{*}$ for $r \ge 2R_{1}$ for some number $R_{1} \gg 6M$. Thus in the intermediate region $(R_{1},2R_{1})$, $\ti r$ involves some smooth monotone interpolant between the usual $r$ coordinate and the Regge-Wheeler coordinate $r^{*}$. We work with $\ti r$ throughout this article, but for simplicity of notation, \textit{henceforth we denote $\ti r$ by either $|x|$ or $r$} (including the statement of the main estimate \cref{eq:main bound} in \cref{thm:main}).

Regarding the (minor) difference between $\ti t$ and the usual time coordinate $t$, we refer the reader to Section 2.1 of \cite{LinToh2} or Section 2 of \cite{LinToh1}.  In the Schwarzschild case $\ti t=t$ for $r>5M/2$. A similar remark holds for the Kerr case, but for Kerr one first performs a removal of singularities at $r_\pminus$. \textit{Henceforth we shall write $t$ to represent $\ti t$.}

	Thus the equation \cref{eq:problem} has been transformed to 
\begin{equation}\label{transformed}
P\p = f(\pa^{2}\phi,\phi,t,x), \quad f := \Hab(t,x)\p\paa\pab\p + O(\p^{2}) \paa\pab\phi.
\end{equation} 
In this article, we shall work with the equation \cref{transformed}.

\subsection*{Notation and local energy norms}
We state some notation that we use throughout the paper.
	We write $X\ls Y$ to denote $|X| \leq CY$ for an implicit constant $C$ which may vary by line. Similarly, $X \ll Y$ will denote $|X| \le c Y$ for a sufficiently small constant $c>0$. 
	The event horizon will play little role in our analysis and so we shall for simplicity of notation henceforth denote $\ti t$ by, more simply, $t$. 
	
	In $\calM$, we consider %
\[
\partial := (\partial_t, \partial_1,\pa_2,\pa_3), \qquad \Omega := (x^i \partial_j -
x^j \partial_i)_{i,j}, \qquad S := t \partial_t + \sum_{i=1}^3x^i \partial_{i},
\]
which are, respectively, the generators of translations, rotations and scaling. We denote the angular derivatives by $\pao$.
	We set
$$Z := (\pa,\Omega,S) $$ and we define the 
 function space
	$$S^Z(f)$$ 
to be the collection of real-valued functions $g$ such that $|Z^J g(t,x)| \ls_J |f|$
whenever $J$ is a multiindex. We will frequently use the upper bounding function $f = \jr^\x$ for some real $\x \leq 0$, where $\jr := (1 + |r|^2)^{1/2}$. We also define the radial subclass $S^Z_\text{radial}(f) := \{ g \in S^Z(f) : g \text{ is spherically symmetric}\}.$
	We denote
\begin{align}  \label{vf defn}
\begin{split}
\p_{J} &:= Z^J\p := \partial^i \Omega^j S^k u, \  \text{ if } J = (i,j,k)\\
\pa^{\le m} \p &:= (\pa^I \p)_{|I|\le m}, \ \text{ where } I \text{ is a 4-index}.%
\end{split}
\end{align}
We use the following $\bar\pa$ notation for the tangential derivatives
\begin{equation}
\label{def:tangential}\bar\pa := \{ \pat+\pa_{r}, \pao\}
\end{equation}
where $\pao$ denotes the angular derivatives. This is used in \cref{thm:iteration}. 

The usual local energy decay estimate is as follows in \cref{LED}. Before we state the estimate, we define the $\leo$ norm. Let
$$A_R := \{ x\in\rt:R \le |x|<2R\} \ \ (R>1), \qquad A_{R=1} := \{ |x|<2 \}.$$
Given a subinterval $I$ of $\R^+$,
\begin{equation}\label{initial.LE.def}
\begin{split}
 \| \phi\|_{LE(I)} &:= \sup_R  \| \la r\ra^{-\frac12} \phi\|_{L^2 (I\times A_R)},\\
  \| \phi\|_{LE^1(I)} &:= \| \de \phi\|_{LE(I)} + \| \la r\ra^{-1} \phi\|_{LE(I)},\\
 \| f\|_{LE^*(I)} &:= \sum_R  \| \la r\ra^{\frac12} f\|_{L^2 (I \times A_R)}.
\end{split} 
\end{equation}
Higher-order versions of \cref{initial.LE.def} are as follows:
\[
\begin{split}
  \| \phi\|_{LE^{1,k}(I)} &= \sum_{|\alpha| \leq k} \| \partial^\x \phi\|_{LE^1(I)} \\
  \| \phi\|_{LE^{0,k}(I)} &= \sum_{|\alpha| \leq k} \| \partial^\x \phi\|_{LE(I)},\\
  \| f\|_{LE^{*,k}(I)} &=  \sum_{|\alpha| \leq k}  \| \partial^\alpha f\|_{LE^{*}(I)}.
\end{split}  
\]
If the subinterval $I$ is omitted, then the norm will involve an integration over all $t\in \R^{+}$, where $\R^{+} := [0,\infty)$.

The following scale-invariant estimate on Minkowski backgrounds is well known and is called an integrated local energy decay estimate (or more simply, local energy decay estimate): we have
\begin{equation}\label{localenergyflat}
\|\pa \p\|_{L^{\infty}_t L^2_x} + \| \p\|_{LE^1}
 \ls \|\pa\p(0)\|_{L^2} + \|\Box \p\|_{LE^*+L^1_t L^2_x}
\end{equation}
and a similar estimate involving the $LE^1[t_0, t_1]$ and $LE^*[t_0, t_1]$ norms.
	See \cite{M} for the case of the Klein-Gordon equation, and see the following works in the case of small perturbations of the Minkowski space-time: see for instance \cite{KSS}, \cite{KPV}, \cite{SmSo}. Even for large perturbations, in the absence of trapping, \eqref{localenergyflat} still sometimes holds, see for instance \cite{BH}, \cite{MST}.
In the presence of trapping, \eqref{localenergyflat} is known to fail, see \cite{Ral}, \cite{Sb}.

We will assume an estimate similar to, but weaker than, \cref{localenergyflat} for our operator $P$ after commuting with \textit{only a finite number of derivatives} (as opposed to vector fields $Z$). In particular we do not assume that we can control the time derivative on the left-hand side---see \cref{eq:SILED}, which is the weaker estimate that we assume. 

 For sake of comparison, we state the usual ILED estimate, see \cref{eq:LED}, but we do not assume \cref{eq:LED} in the present article. For the Kerr spacetime with large $a$, the estimate \cref{eq:LED} holds for the homogeneous wave equation (by this we mean that in the notation of \cref{eq:LED} we have $P\p=0$): see Theorem 3.2 in \cite{DRS}.
 	The reader is encouraged to compare \cref{eq:LED} with the base case \cref{localenergyflat}.
\begin{definition}[ILED] \label{LED}
We say that $P$ has the integrated local energy decay
property if the following estimate holds for a finite number $m \geq 0$, and for all $0\leq T_0<T_1 \leq \infty$:
\begin{equation}
 \| \pa^{\le m} \p\|_{LE^{1}([T_0,T_1)\times\rt)} 
\ls_m  \|\pa \p(T_0)\|_{H^m(\rt)} + \|\pa^{\le m}(P\p)\|_{(\lolt + LE^*)([T_0,T_1)\times\rt)}
\label{eq:LED}\end{equation}
where the implicit constant does not depend on $T_0$ and $T_1$.
\end{definition}

We shall assume that the problem \cref{eq:problem} satisfies the following weaker version of the uniform boundedness of the energy, where we allow losses on the right hand side:
\begin{definition}[Weak energy bounds] \label{def:EB}
	 		We will assume that \eqref{eq:problem} satisfies the following estimate: there exists some $k_0 \in \N$ such that for finitely many $m \in\N$,
\begin{equation}\label{EB}
 \| \pa\p(T_{1}) \|_{H^{m}( \rt)} 
\ls_m  \|\pa\p(T_0)\|_{H^{m+k_{0}}(\rt)}, \quad 0 \le T_{0} \le T_{1}.\footnote{The assumption \cref{EB} is used in \cref{pdb1}.}
\end{equation}
\end{definition}

Given a norm $\| \|$, let $\|(f_1,\dots,f_n)\| := \sum_1^n \|f_j\|.$ We will assume that the equation \cref{eq:problem} satisfies the following estimate:
\begin{definition}[SILED]\label{def:SILED}
We say that $P$ has the stationary integrated local energy decay property if the following estimate holds for some $m \ge 0$, and for all $0 \le T_0 < T_1 \le\iy$:
\begin{align}\label{eq:SILED}
\begin{split}
 \| \pa^{\le m}\p\|_{LE^1([T_0,T_1)\times\rt)} 
&\ls_m  \|\big(\pa \p(T_0), \pa\p(T_{1})\big)\|_{H^{m}(\rt)} + \|\pa^{\le m}(P \p)\|_{ LE^*([T_0,T_1)\times\rt)} \\
&\qquad + \|\pa^{\le m}(\pat\p)\|_{LE([T_0,T_1)\times\rt)}.
\end{split}
\end{align}
\end{definition}

\begin{remark}[The estimate \cref{eq:SILED} holds for the problem \cref{eq:problem}] We note that \cref{eq:SILED} is a statement that an integrated local energy decay estimate holds for the \textit{linear problem} involving either the Schwarzschild or Kerr metric. In the present article we shall show how to control the quasilinearity after it has been placed into the dual local energy, or $LE^*$, norm. 

While the articles \cite{LinToh1,LinToh2} proved what is known as a weak ILED estimate for the equation \cref{eq:problem}, we choose to work with \cref{eq:SILED} instead; \cref{eq:SILED} holds for the equation \cref{eq:problem}, with an easier proof than the proof of weak ILED in \cite{LinToh1,LinToh2}. %

	 In general, SILED is far easier to prove than weak ILED if $\pat$ is timelike near the trapped set (and it is indeed timelike for the well-known applications of this estimate, such as perturbations of the Kerr spacetime).  
		The main obstruction in proving such ILED-type estimates on black hole spacetimes are trapping-related issues, and SILED does not even require decay in time of the metric near the trapped set; see Theorem 4.3 in \cite{MTT} for a demonstration of this claim.
		
 By contrast, proving weak ILED at present requires some decay in time of the metric $g$ toward $g_{B} \in \{ g_{\text{Schwarzschild}}, g_{\text{Kerr}} \}$; in \cite{MTT} they assume $\eps t^{-1-\x}$ ($\x>0$ arbitrarily small) decay while this was improved to an assumption of only $\eps t^{-1/2}$ decay in the articles \cite{LinToh1} and \cite{LinToh2}, but the latter is nonetheless still an assumption of time decay. %
\end{remark}

See also \cref{app} for how the change to normalised coordinates mentioned earlier affects the ILED and energy norms.

\subsection*{%
Pointwise bound assumptions on $\p$ and $H$, and thus on the metric %
}

Let $N$ be a sufficiently large number that equals the number of vector fields we assume on the initial data. 
	Let $N_{1}\le N/2 + 2$.\footnote{This particular  value of $N_1$ is merely an artifact of the proof in \cite{LinToh1,LinToh2}, and we note that its value can be changed to $N_1 := N - j$ for some $j \in \N$ without difficulty.}

An immediate consequence of the finiteness of the $\leo$ norm (which follows from any of various ILED-type estimates such as \cref{eq:SILED}) is
\begin{equation}\label{assu}
\sum_{J:|J| \le N_{1}} |\p_{J} | \ls  \la t - r^{*} \ra^{1/2}\jt\inv, \qquad \sum_{J:|J| \le N_{1}} |\pa \p_{J}|  \ls  \jr\inv \la t-r^{*} \ra^{-1/2}. 
\end{equation} 
The authors in \cite{LinToh1,LinToh2} demonstrated finiteness of the $\leo$ norm given the assumption of sufficiently small initial data that are smooth and compactly supported. 
	\, In fact, by smallness of the initial data we have $\eps$ factors on the right-hand sides of \cref{assu}, but we do not need this smallness in the estimates \cref{assu} in the present article. 

Regarding assumptions on the functions $\Hab(t,x)$, for all our propositions, lemmas and theorems in the present article we assume only that for some finite number $N$,
\begin{align}\label{met.ass}
\begin{split}
\sum_{J: |J| \le N}|H_J|&\ls 1, \quad  \sum_{J: |J| \le N} |H^{\Lbar\Lbar}_{J}|\ls \ju^{\xd}\jt^{-\xd}\\
\sum_{J: |J| \le N} |\pa H_J| &\ls \jt/(\jr\ju).
\end{split}
\end{align}See also \cref{app} (the appendix) for more context to the assumption on $H^{\Lbar\Lbar}$. 
We are able to assume a weaker but more complicated-looking assumption on $\pa H_{\le N}$, but for simplicity of presentation we simply assume that $\pa H_{\le N}\ls \jt/(\jr\ju)$. %

The proofs of global existence for \cref{eq:problem} in \cite{LinToh1} and \cite{LinToh2} for the case that $g_B$ equals the Schwarzschild and Kerr metric respectively make further assumptions on the metric coefficients, beyond \cref{met.ass,assu}. We refer the reader to those articles for the complete list of assumptions. 
	Here we only state assumptions that we need for our pointwise decay result in \cref{thm:main}, namely only \cref{met.ass,assu}. %

We now state our main theorem:

\begin{theorem}\label{thm:main}
\begin{enumerate}
\item (The exact Schwarzschild and Kerr metrics) Let $\phi$ denote the unique global solution of \cref{eq:problem} with small, smooth and compactly supported initial data with $g_B$ equal to either the Schwarzschild or Kerr metric, and %
with the assumptions \cref{met.ass,assu,EB,eq:SILED}.
	Then the following bound holds in normalised coordinates \cref{P def}%
\begin{equation}\label{eq:main bound}
\sum_{J: |J|=0}^m |\p_J(t,x)| \ls \f1{\la t+|x|\ra\la t-|x|\ra^{\xk}},  \quad \xk := \min(\xd,1), \quad |x| > r_e.
\end{equation}
This implies that for bounded $|x|$ with $|x| >r_e$, we have
$$|\p(t,x)| \ls \jt^{-1-\min(\xd,1)}.$$

\item (More general asymptotically flat spacetimes)
Assume that global existence holds for solutions to $$P'\p = f(\pa^2\phi,\phi,t,x)$$ %
where
\begin{equation}
\label{P' def}P' \in \Box + s_{2+}^\xo \Delta_\xo + \paa s_{1+}^{\x\xb} \pab + s_{2+} + s_{1+}^\x \paa
\end{equation}which describes more general asymptotically flat spacetimes than $P$ does, 
where we use the summation convention and the functions $s_\eta = s_\eta(t,x)$ satisfy $s_\eta \in S^Z(\jr^{-\eta}), \eta \in \{ 1+,2+ \}$, and we assume the coefficients of $\Delta_\xo$ are radial. Assume that \cref{met.ass,assu,EB,eq:SILED} hold with $P'$ replacing $P$ in \cref{eq:SILED}. 
	We also make the standard assumption that the metric described in \cref{P' def} is a smooth Lorentzian metric on $\R_t \times \R^3 \setminus B(0,r_0)$ for some $r_0>0$. %

Then the same bound \cref{eq:main bound} holds for these solutions, and thus as a consequence, we again have
$$|\p(t,x)| \ls \jt^{-1-\min(\xd,1)}, \quad |x| > r_e.$$
\end{enumerate} 
\end{theorem}%

\subsection{Outline of the paper} \label{ss:outline}

In \cref{sec:notation} we define notation that is used throughout the article.
In \cref{sec:fromLEDtoptw} we state initial pointwise estimates. 
	In \cref{sec:preliminaries} we prove \cref{conversion}, which is the core result in our iteration scheme. This lemma adopts the perspective of rewriting the equation in order to control $\p=\Box_\text{Mink}\inv g$, where 
\begin{itemize}
\item
$g = g(P-\Box_\text{Mink}
	,\p, \pa^2\p)$ has coefficients (see \cref{P def} and \cref{met.ass}) that decay more rapidly %
than the coefficients of the Minkowski metric,
\item
and for the terms involving $\p$ or derivatives of $\p$ in $g$, the aforementioned \cref{sec:fromLEDtoptw} gives us starting pointwise estimates for $\p$ and derivatives of $\phi$. 
\end{itemize}That is, $g$ obeys good estimates. 
	Using $\Box_\text{Mink}\inv$ in this way allows us the simplification of integrating on a smaller set than mere finite speed of propagation would allow for. %
	Hence by using $\Box_\text{Mink}\inv$ we can integrate these good estimates that $g = g(\text{coefficients}, \p, \pa^2\p)$ obeys on a ``small'' set $D$ (see \cref{def:Dtr} for this set), and in $D$ we are able to avoid integrating in a ``small $r$ and small $t$'' region. Thus in our integrations, small $r$ values are paired with large $t$ values\footnote{and small $t$ values are paired with large $r$ values (by this we mean that, say, $r > 3t/2$), but this does not appear in this problem because we assume compactly supported initial data, as is clear after taking into account finite speed of propagation for the equation.}.  This perspective was also adopted in previous works such as \cite{L,Loo22.II,MTT}. 

The statement \cref{eq:SILED} can be thought of as saying that the spatial portion of the operator $P$ (by this we mean the terms of $P$ that remain if we ``delete'' all terms that have either a single time derivative $\pat$ or two time derivatives $\pat^2$)
 is invertible in the local energy norm. 
 	In \cref{sec:propag} we use the assumption \cref{eq:SILED} to propagate pointwise decay proved in the region $\{ r \ge t/2\}$ (this uses the tools in \cref{sec:preliminaries}) into the interior region $\{ r \le t/2\}$. 
	That is to say, if $\p$ obeys the bounds $\p \ls \jr\inv F$ for some function $F = F(|t-r|)$, we are able to use \cref{eq:SILED} to prove the bounds $\p \ls \jt\inv F$. 

In \cref{sec:iter} we prove the final decay rate for $\phi$ and its vector fields in the region inside of the light cone, that is, $\{ r \leq t\}$, which thereby completes the proof because of finite speed of propagation. In \cref{app} we provide an appendix giving more details about the Schwarzschild and Kerr metrics and we give more context to the assumption on $H^{\Lbar\Lbar}$; we also remark that the ILED and energy norms are minimally affected by the change to normalised coordinates.

\section{Notation}\label{sec:notation}

\subsection*{Notation for dyadic numbers and conical subregions}
We work only with dyadic numbers that are at least 1. We denote dyadic numbers by capital letters for that variable; for instance, dyadic numbers that form the ranges for radial (resp. temporal and distance from the cone $\{|x|=t\}$) variables will be denoted by $R$ (resp. $T$ and $U$); thus $$R,T, U\ge 1.$$ 
	We choose dyadic integers for $T$ and a power $a$ for $R,U$---thus $R = a^k$ for $k\ge1$--- different from 2 but not much larger than 2, for instance in the interval $(2,5]$, such that for every $j\in\N$, there exists $j'\in\N$ with 
$a^{j'} = \f38 2^j.$
 
\subsubsection*{Dyadic decomposition}
We decompose the region $\{r\le t\}$ based on either distance from the cone $\{r=t\}$ or distance from the origin $\{r=0\}$. We fix a dyadic number $T$. 
\begin{align*}
C_T &:= \begin{cases}
\{ (t,x) \in [0,\iy) \times \rt : T \leq t \leq aT, \ \ r \leq t\} & T>1 \\
\{ (t,x) \in [0,\iy) \times \rt : 0 < t < a, \ \ r \leq t\} & T=1
\end{cases} \\
C^R_T &:=\begin{cases}
C_T\cap \{R<r<aR\} & R>1\\
C_T\cap \{0 < r < a\} & R=1
\end{cases}\\
C^U_T &:=\begin{cases}
\{ (t,x) \in [0,\iy) \times \rt  :  T\le t\le aT\} \cap \{U<|t-r|<aU\} & U>1\\
\{ (t,x) \in [0,\iy) \times \rt  :  T\le t\le aT\} \cap \{0< |t-r|<a\} & U=1
\end{cases} %
\end{align*}
	We define
\begin{equation*}
C_T^{<3T/4} := \bigcup_{R < 3T/8} C_T^R.
\end{equation*}

	$C_T^R, C_T^U$ are where we shall apply Sobolev embedding, which allows us to obtain pointwise bounds from $L^2$ bounds. 
Given any subset of these conical regions, a tilde atop the symbol $C$ will denote a slight enlargement of that subset on its respective scale; for example, $\crtt$ denotes a slightly larger set containing $\crt$.

\subsection{Notation for the symbols $n$ and $N$} \label{subsec:N}
Throughout the paper the integer $N$ will denote a fixed and sufficiently large positive number, signifying the highest total number of vector fields that will ever be applied to the solution $\p$ to \eqref{eq:problem} in the paper. 

We use the convention that the value of $n$ may vary by line.%

If $\xS$ is a set, we shall use $\ti\xS$ to indicate a slight enlargement of $\xS$, and we only perform a finite number of slight enlargements in this paper to dyadic subregions. The symbol $\ti\xS$ may vary by line. 

\subsection*{Summation of norms}\label{conv:normsum}
\begin{itemize}
\item
We write 
$$\|(f_1,\dots,f_n)\| := \sum_1^n \|f_j\|.$$
\item
Recall the subscript notation \cref{vf defn} for vector fields. Let $\| \cdot\|$ be any norm used in this paper. Given any nonnegative integer $N\ge0$, we write $\|g_{\le N}\|$ to denote $\sum_{|J|\le N} \|g_J\|$. For instance, taking the absolute value as an example of the norm, the notation $|\pm(t,x)|$ means
$$|\pm(t,x)| = \sum_{J : |J| \leq m} |\p_J(t,x)|.$$
\end{itemize}

\subsection*{Other notation}
If $x =(x^1,x^2,x^3)\in\R^3$, we write 
\begin{align*}
u:= \f{t-r}2,  \quad v := \f{t+r}2.
\end{align*}
We write $\Box := -\pat^2+\Delta$.

\begin{definition}[Backward light cone]\label{def:Dtr}

Let $D_{tr}$ denote   
	\[
	D_{tr} := \{ (\rho,s)  \in \R_+^2: -(t+r) \leq s-\rho\leq t-r, \ |t-r| \leq s+\rho \leq t+r\}.
	\] 
When we work with $D_{tr}$ we shall use $(\rho,s)$ as variables, and $D_{tr}^{ R}$ is short for $D_{tr}^{\rho\sim R}$.
Thus for $R>1$, let 
\begin{equation*}
\co:= D_{tr} \cap \{ (\rho,s) : R < \rho<2R\}
\end{equation*}
and let 
$$D^{R=1}_{tr} := \dtr\cap\{(\rho,s) : \rho<2 \}. $$
We write $dA := dsd\rho.$

Let 
$$\calR_1:= \{ R : R < (t-r)/8 \}, \quad\calR_2 := \{ R : (t-r)/8 < R < t+r \}.$$ The partition $\calR_{j}$ will be used in the context of the pointwise decay iteration located in \cref{sec:iter}, which uses an integration over the backward light cone $D_{tr}$.
\end{definition}

\section{Initial pointwise estimates}\label{sec:fromLEDtoptw}

In this section we will show that local energy decay bounds imply certain slow decay rates for the solution, its vector fields, and its derivatives---see \cref{inptdcExt,derbound}.

The following pointwise estimate for the second derivative in \cref{2ndDeBd} (this was proved in \cite{LinToh1}) will be used, for instance, when applying \cref{DyadLclsd} to the functions $ w = \pa\pm$ (that is, when we bound the first-order derivatives pointwise); this will be done in \cref{derbound}. 

Let $R_{1}>0$ be a large number. %
This number satisfies the property that one multiplies the equation with the Morawetz multiplier only in $r\ge R_{1}$. We prove the following lemma only in %
$r \ge 2R_{1}$. %

See also \cite{L,Loo22.II,LooToh22,Loo22} for other proofs of this result in a very similar setting. We sketch the main points of the proof in this setting, but see Lemma 6.3 in \cite{LinToh1} for a more complete proof. 

\begin{lemma}\label{2ndDeBd}Assume that $\p$ is sufficiently smooth. 
For any point $(t,x)$ with $|x| \ge 2R_{1}$ we have 

 \begin{equation}\label{2ndD}
 \pa^2\p_J \ls
\mu\inv |\pa\p_{\le|J|+1}| + \mu\inv\jr\inv |\p_{\le |J|+2}| %
\quad \mu :=\min(\jr,\ju).
 \end{equation}
\end{lemma}

\begin{proof}We assume the stationary ILED bound \cref{eq:SILED}.
We prove this claim with a two-step process, as indicated by items (1) and (2) below: first for multi-indices $|J| \le N/2$, and then for $|J| \le N$.
\begin{enumerate}
\item\label{Item1}
For $|J|\le N/2$, %

We begin with the case $|J|=0$. 
	Since $\pao \sim r\inv \xO$ and $\pat = t\inv(S - r\pa_{r})$, we have 
$$\pa^{2}\p \ls r\inv |\pa\p_{\le3}| + |\pa_{r}^{2}\p|.$$
We also have 
$$\pa_{r}^{2}\p \ls \ju\inv |\pa\p_{\le3}| + t\ju\inv |\Box\p|.$$
	We now transition the $\Box$ term into a $\Box_{g}$ term. We write $|\Box\p| \le |\Box_{g}\p| + |(\Box_{g}-\Box)\p|$. For simplicity of notation, we shall write $\Box_{B}:= \Box_{g_{B}}$. Note that 
\begin{align*}
|(\Box_g -\Box) \p| &\ls |(\Box_{g} - \Box_{B})\p| + |(\Box_{B} - \Box)\p| \\
	&\ls |h \pa^{2}\p| + |\pa h \pa\p| + r\inv |\pa \p_{\le1}|.
\end{align*}
	Thus by \cref{assu} we have 
$$t\ju\inv\Box\p \ls t\ju\inv|\Box_{g}\p| + \eps\ju^{-1/2}|\pa^{2}\p, tr\inv\ju^{-1/2}\pa\p| + t r\inv\ju\inv|\pa\p_{\le3}|.$$
Let $\mu := (r\ju)/t$. Thus
$$\pa_{r}^{2}\p\ls \mu\inv |\pa\p_{\le3}| + r\mu\inv \cdot |\Box_g\p|.$$
In conclusion,
$$\pa^{2}\p\ls \mu\inv |\pa\p_{\le3}| + r\mu\inv \cdot |\Box_{g}\p|.$$

The proof by induction for higher $0 < |J| \le N/2$ is similar, and also uses \cref{assu}. 

\item %
For $|J|\le N$, we now show
\begin{equation}
\pa^{2}\p_{J}\ls \mu\inv |\pa\p_{\le|J|+3}| + \mu^{-2}|\p_{\le|J|}| + r\mu\inv |( \Box_{g}\p)_{J}|.
\end{equation}

Let 
	$$W^{\xb}:= |g_B|^{-1/2}\paa(\gab_B |g_B|^{1/2} )- |g|^{-1/2}\paa(\gab |g|^{1/2} ).$$
Let 
$$R_{1}^{*} := r^{*}(R_{1}). $$
A quick consequence of the assumptions is (see Lemma 5.2 in \cite{LinToh1}) that
\begin{equation}
W_{J} \in S^{Z}(1) \pa\p_{J} + \text{Rem}_{J}
\end{equation}
where 
\begin{equation}
\label{RemBd}%
\text{Rem}_{J} \ls \begin{cases}
|\p_{\le J}| & r \le R_{1}\\
|\pa\p_{\le J} + (r^{-2} + t/(r \la t-r^{*}\ra) ) |\p_{\le J}| & r^{*} \ge R_{1}^{*}
\end{cases}
\end{equation}
where the $t/(r\la t-r^{*}\ra)$ bound is simply the upper bound assumed on $\pa \Hab$ in \cref{met.ass}. 
	We have 
$$(\Box_{g}\p - \Box_{B}\p)_{J} = \sum_{J_{1} + J_{2}=J} \hab_{J_{1}} \pa_{\x\xb}\p_{J_{2}} + \sum_{J_{1} + J_{2}=J} W^{\x}_{J_{1}} \paa \p_{J_{2}}.$$
Either $|J_{1}|\le N/2$ or $|J_{2}|\le N/2$: \begin{itemize}
\item
If $|J_{1}|\le N/2$, then by \cref{assu}, \cref{Item1} and \cref{RemBd}
$$h_{\le N/2} \ls \eps \ju^{1/2}\jt\inv, \quad W_{\le N/2}\ls \jr\inv\ju^{-\xd}$$
we have
$$t\ju\inv |\hab_{J_{1}} \pa_{\x\xb}\p_{J_{2}}| \ls \eps \ju^{-1/2}(|\pa^{2}\p_{J}| + |\pa^{2}\p_{\le |J|-1}|)$$
$$t\ju\inv |W^{\x}_{J_{1}} \paa \p_{J_{2}}| \ls \mu\inv |\pa \p_{\le|J|}|.$$
\item
If $|J_{2}|\le N/2$, then by \cref{Item1} above, 
$$\pa\p_{\le N/2}\ls \eps \jr\inv\ju^{-1/2}, \quad \pa^{2}\p_{\le N/2}\ls \eps \mu\inv \cdot r\inv \ju^{-1/2}$$
and hence
$$t\ju\inv |\hab_{J_{1}} \pa_{\x\xb}\p_{J_{2}}| \ls \eps \mu^{-2}\ju^{-1/2}|\p_{\le|J|}|$$
\begin{align*}
t\ju\inv |W^{\x}_{J_{1}} \paa \p_{J_{2}}| 
  &\ls \mu\inv \ju^{-1/2}|W_{\le |J|}|\\
  &\ls \mu\inv \ju^{-1/2} |\pa\p_{\le|J|}| + \mu^{-2} |\p_{\le|J|}|.
\end{align*}The final line follows form the $r^{*} \ge R_{1}^{*}$ part of \cref{RemBd}. 
\end{itemize}The conclusion now follows.%
\end{enumerate}
\end{proof}

\begin{remark}[More general asymptotically flat spacetimes]
\cref{2ndDeBd} also holds for operators $P'$ defined in \cref{P' def} with coefficients that have even slower decay rates than the one defined in \cref{P def}.

Similarly, all the other propositions, lemmas and theorems involving $P$ in this article also hold if $P$ is replaced by $P'$ but for the sake of simplicity of presentation, in the remainder of this article we mainly work with the operator $P$. The proofs for $P'$ are straightforward and can be found in \cite{Loo22.II} or \cite{L}. 
\end{remark}

\begin{remark}
Henceforth we shall make no more mention of the trapped set, and we note that all the results below still hold outside of a ball centred at $0 \in \R^3_x$ with only very minor modifications of the proofs.
	Below, even though we superficially work on $\R^3_x$, our proof implicitly takes place only in $\R^3_x \setminus B(0,r_e )$.
\end{remark}

The primary estimates that let us pass from local energy decay to pointwise bounds are contained in \cref{DyadLclsd}.
\begin{lemma}
\label{DyadLclsd}
Let $w\in C^4$,  $Z_{ij} :=S^i\xO^j$,  $\mu:= \la \min(r,|t-r|) \ra$, and $\calR \in \{\crt,\cut\}$. 
	Then we have  \begin{equation}\label{DyadLclsdBd}
 \| w\|_{L^\iy(\calR)} \ls \sum_{i\le 1,j\le 2} \f1{|\calR|^{1/2}}  \lr{  \|Z_{ij} w\|_{L^2(\calR)} +  \|\mu \pa Z_{ij} w\|_{L^2(\calR)} }.
\end{equation}
where we assume $1\ll U \le \f38 T$, $1\ll R \le \f38 T$ in the cases $\cut,\crt$ respectively, and $|\calR|$ denotes the measure of $\calR$. 
\end{lemma}

\begin{proof}
A change of coordinates into exponential coordinates results in $\calR$ being transformed into a region of constant size in all directions. Then one uses the fundamental theorem of calculus for the $s,\rho$ variables and Sobolev embedding for the angular variables. Finally, changing coordinates to return to the original region $\calR$ produces the $|\calR|^{-1/2}$ factor. 
\end{proof}

The following \cref{Hardy} will be used to prove \cref{inptdcExt}. 

\begin{lemma}\label{Hardy}
If $f\in C^1( [0,\iy)_t \times \R^3_x )$, then
\begin{equation}\label{con.Hardy}
\int_{t/2}^{3t/2} \f{f(t,x)^2}{\ju^2} dx \ls \int_{t/4}^{7t/4}|\pa_r f(t,x)|^2 dx + \jt^{-2} \left(\int_{t/4}^{t/2} f(t,x)^2 dx + \int_{3t/2}^{7t/4} f(t,x)^2 dx\right)
\end{equation}
\end{lemma}

The next proposition yields an initial global pointwise decay rate for $\phi_J$ under the assumption that the local energy decay norms are finite. We shall improve this rate of decay in future sections (see \cref{sec:iter}) for solutions to \eqref{eq:problem}, culminating ultimately in the final pointwise decay rate stated in the main theorem (\cref{thm:main}). 

\begin{proposition} \label{inptdcExt}
Let $T$ be fixed and $\p$ be any sufficiently regular function.
There is a fixed positive integer $k$, %
such that for any multi-index $J$ with $|J|\le N - k$, we have:
\begin{align}\label{u/v decay}
\begin{split}
|\p_{J}| \leq \bar C_{|J|} \|\p_{\leq |J|+k}\|_{LE^1[T, 2T]} \f{\ju^{1/2}}{\jv}. %
\end{split}\end{align} %
\end{proposition}
\begin{proof}
One uses \cref{DyadLclsd}, which proves \cref{u/v decay} except in the wave zone $\cut$. For $\cut$ we additionally use the tool of \cref{Hardy}, because bounding directly using only \cref{DyadLclsd} results merely in the insufficient decay 
$$\|\p_{J}\|_{\licut}\ls U^{-1/2}\|\pmn\|_{\leo[T,2T]}.$$
\end{proof}

\begin{lemma}\label{lem:Sob}
Given a function $f : \R^{3}\to \R$, we have 
$$\|f\|_{L^{\iy}(R < |x| < R+1)} \ls R\inv \|f_{\le 3}\|_{\lt(R - 1 < |x| < R+2)}.$$
\end{lemma}

\noindent \textit{Sketch of proof.} 
This standard result can be proven by combining a localised embedding and an embedding on $S^{2}$. \hfill $\Box$

We shall now state two pointwise bounds for the derivative, namely \cref{pdb1,derbound}, one an automatic consequence of the uniform energy bounds, and the other a consequence of bootstrapping from already-existing pointwise bounds on vector fields of the solution.

\begin{corollary} \label{pdb1}
We have\begin{equation}\label{cor:Sob}
\pa \pm \ls \jr\inv
\end{equation}where the constant depends on the initial data.
\end{corollary}
\begin{proof}
This is an immediate corollary of \cref{lem:Sob} and the uniform boundedness of the energy \cref{EB}.
\end{proof}

\begin{lemma}[The derivative in $\lt$]\label{lem:derbd}
For $\calR\in\{\crt,\cut\}$ and for $R\gg1, U\gg1$ respectively, then
\begin{equation} \label{easy}
  \| \pa \p\|_{\lt(\calR)} 
  	\ls  \|\mu\inv \p_{\le1}\|_{\lt(\ti\calR)}.
\end{equation}

\end{lemma}

\begin{proof} We only prove the case $\crt$ as the other two cases are similar; see also Lemma 5.1 in \cite{L}.
Let $\chi(t,r)$ be a radial cutoff function on $\R^{1+3}$ with $\supp\chi \subset \crtt$ and $\chi=1$ on $\crt$. 

Let $w$ be a function. Note that
\begin{enumerate}
\item 
If $r < t$ then for a sufficiently large constant $C'$, we have %
\begin{equation}\label{1f}
\chi \lr{ \f{u}t |\de w(t,x) |^2  }\le \chi\lr{   |\grad w |^2- w _t^2+ \f{C'}{ut} |S w |^2 }
\end{equation}
as an expansion of the terms $|Sw|^2, |\de w|^2$ reveals.%

\item Integrating by parts,
\begin{align}\label{ini.com}
\int \chi ( |\grad w |^2 -  w _t^2) \, dxdt
	&=  \int \chi w (\pat^2-\Delta) w \,dxdt - \int \f12 (\pat^2-\Delta)\chi  w ^2\, dxdt.
\end{align} There are no boundary terms in either time or space because of the compact support of $\chi(t,r)$ in both time and space. 
\end{enumerate}
 	Integrating \cref{1f} in spacetime, we have via \cref{ini.com}
\begin{align}\label{1g}
\int\chi \f{u}t |\de w |^2 \,dxdt
	&\le \int \chi w(\pat^2-\Delta) w + O(|\Box\chi| w ^2) + \f{C'}{ut}\chi |S w |^2 \,dxdt.
\end{align}
For 
\begin{equation}
\label{comeback}
\int (\chi w)(\hab_{\text{B}}\paab w) \,dxdt,
\end{equation} 
we integrate by parts and use Cauchy-Schwarz. A term
$$\int \chi \hab_{\text{B}} \paa w \pab w \,dxdt = O\lr{ \int \chi \f{ |\de w|^2}\jr \,dxdt }$$
arises, and for this term we use the hypothesis that $R \gg 1$.  For the other term that arises, 
$$\int \chi \pa h_{\text{B}} w \cdot \pa w \, dxdt$$
we use that $\pa h_{\text{B}} \in S^{Z}(\jr\inv)$, note that $\jr \gg1$, and then use Cauchy-Schwarz: for a small $\eps'>0$, 
$$\eps' \int \chi \cdot \f{u}t \cdot |\de w|^{2}dxdt + \f1{\eps'} \int \chi \cdot \f{t}u \cdot \jr^{-2}w^{2} \,dxdt.$$
We can absorb the left hand term to the left-hand side. 

Assuming $\Box\chi\ls \jr^{-2}$, separating
$|\chi wPw| \ls \chi [ (R\inv w)^2 + (R Pw)^2 ]$ in the right-hand side of 
\cref{1g}, and using the triangle inequality to deal with \cref{comeback}, this proves the desired claim \cref{easy}.
	Thus by \cref{1g} we conclude
\begin{equation*}
 \| \pa \p\|_{\lt(\calR)} 
  	\ls  \|\mu\inv \p_{\le1}\|_{\lt(\ti\calR)} +  \|\jr (\Hab \p \paab\p + \p^{2} \paab\p) \|_{\lt(\ti\calR)}
\end{equation*}and the triangle inequality now establishes \cref{easy}.
	The proof of \cref{easy} for $\cut$ is similar except we now note that $\pa \chi \ls \jt\inv\ju\inv$. 
	
	To conclude the proof of \cref{easy} we observe that 
	$$\jr\Hab\p\pa^{2}\p \ls \mu\inv |\p_{\le 1}|$$
by \cref{2ndDeBd,cor:Sob}. Similarly,
	$$\jr O(\p^{2})\pa^{2}\p \ls \mu\inv |\p_{\le 1}|.$$
\end{proof}

\begin{lemma}\label{cor:derbd}
For $\calR\in\{\crt,\cut\}$
\begin{equation} \label{easy}
  \| \pa \phi_{\le m}\|_{\lt(\calR)} 
  	\ls  \|\mu\inv \phi_{\le m+n}\|_{\lt(\ti\calR)}.
\end{equation}
\end{lemma}

\begin{proof}
Compared to the proof of \cref{lem:derbd}, here one only has to bound $$\int \chi w_{\le m} [P,Z^{\le m}] w \,dxdt.$$ 
Similar arguments involving integration by parts and Cauchy-Schwarz as those seen in \cref{lem:derbd} establish that
$$  \| \pa w_{\le m}\|_{\lt(\calR)} 
  	\ls  \|\mu\inv w_{\le m+n}\|_{\lt(\ti\calR)} +  \|\jr  (\Hab \p \paab\p + O(\p^{2}) \paab\p)_{\le m}\|_{L^2(\tilde \calR)}.$$
	Then to conclude the proof, we simply observe that by \cref{2ndDeBd,cor:Sob}, we have
   $$(\jr\Hab\p\paab\p)_{\le m} \ls  \mu\inv|\pmn|.	$$
  Similarly,
  $$(\jr O(\p^{2}) \paab\p)_{\le m} \ls \mu\inv|\pmn|.$$
	The proof is complete.
\end{proof}

The next proposition shows that the first-order derivative (of solutions to \eqref{eq:problem}) decays pointwise faster by a rate of $\min(\jr,\nm)$. It utilises the initial global decay rate \cref{u/v decay} and the $\lt$ bound \cref{easy} for the first order derivative. 
\begin{proposition} \label{derbound}
Let $\p$ solve the equation \cref{eq:problem}, and assume that 
$$\pmn \ls \jr^{-\x}\jt^{-\xb}\ju^{-\eta}$$
for some sufficiently large $n$. We then have
\begin{equation}\label{claim}
\pa\pm \ls \jr^{-\x}\jt^{-\xb}\ju^{-\eta} \mu\inv, \quad \mu := \la \min(r,| t-r |) \ra.
\end{equation}
\end{proposition}

\begin{proof} Let $\calR \in \{ \cut, \crt\}$. 
	Recalling \cref{DyadLclsd}, we have 
\begin{align*}
\|\pa\pm\|_{L^\iy(\calR)} 
&\ls |\calR|\invh \sum_Z \| (Z\pa \pm, \mu \pa Z \pa \pm) \|_{\lt(\calR)} \\
&\ls |\calR|\invh \lr{ \|\pa\pmn\|_{\lt(\calR)} + \|\mu \pa^2 \pmn\|_{\lt(\calR)}	} \\
&\ls |\calR|\invh \lr{ \|\mu\inv \pmn\|_{\lt(\ti\calR)} + \|\mu \pa^2 \pmn\|_{\lt(\calR)}	} \\
&\ls |\calR|\invh \lr{ \|\mu\inv \pmn\|_{\lt(\ti\calR)} + \| \mu \left( \f1\mu|\pa\pmn| + ( 1 + \f{t}\ju ) \jr^{-2}|\pmn| \right) \|_{\lt(\calR)}	} \\
&\ls |\calR|\invh \lr{ \|\mu\inv \pmn\|_{\lt(\ti\calR)} + \|\mu(1 + \f{t}\ju)\jr^{-2} \pmn \|_{\lt(\calR)} } \textrm{by \cref{easy}}\\
&\ls |\calR|\invh \|\mu\inv \pmn\|_{\lt(\ti\calR)} \end{align*}The last line follows because $\mu^2(1 + t/\ju) \ls \jr^2$. 
The claim now follows.
\end{proof}

\section{Preliminaries for the pointwise decay iteration}\label{sec:preliminaries}
\begin{remark}[The initial data] \label{rem:id}
Let $w$ denote the solution to the free wave equation with initial data $(\p_{0}(x), \p_{1}(x))$.
	By the assumption that the initial data are compactly supported and smooth, we have
$$ w_J \ls \jv\inv\ju^{-\min(\xd,1)},$$
which is the final decay rate in \cref{thm:main}.
\end{remark}

\subsection{Summary of the iteration} \label{outline of iteration}
By \cref{rem:id}, we may assume zero initial data in the following iteration. Second, note that it suffices to prove bounds in $\{ u>1\}$, because the desired final decay rate in $t-1 < r < t+C$ already holds by \cref{u/v decay}. 
Third, we distinguish the nonlinearity and the coefficients of $P - \Box$, and for both of these, we apply the fundamental solution. We iterate these two components in lockstep with one another.  

In the iteration process, the decay rates obtained from the fundamental solution are insufficient in the region $\{ r < t/2\}$, so we prove \cref{thm:r-t}. With the new decay rates obtained from \cref{thm:r-t}, we are then able to obtain new decay rates for the solution and its vector fields. At every step of the iteration, \cref{conversion} is used to turn the decay gained at previous steps into new decay rates. 

\subsection{Iteration scheme setup}\label{settingup}
We only work with the assumptions from the $g_B$ case from part (1) of \cref{thm:main} rather than the more general perturbations from part (2) because part (2) follows by straightforward modifications of the following; see for instance \cite{Loo22.II} for the iteration scheme with the more general decay rates stated in part (2).  

We now rewrite the equation for $P\p=f$ and note that the decay rates below are only minimally modified for $P'\p = f$. 
We rewrite \eqref{eq:problem} as
$$\Box\p = (\Box - P)\p + f = -\pa_\alpha(h^{\alpha\beta}\pa_\beta\p) - g^\xo \Delta_\xo \p - V\p + f,$$
with
$f := \Hab \p \paa\pab \p + O(\p^{2}) \paa\pab\phi$. 

Using \eqref{P def}, we can write this as
$$\Box\p \in \pa \left(S^Z(\jr^{-2}) \p_{\leq 1}\right) + S^Z(\jr^{-3}) \p_{\leq 2} + f$$
	Pick any multiindex $|J| \ll N$. We commute with the vector field $Z^J$ and obtain
\begin{equation}\label{first write}
\Box\p_J \in \pa \left(S^Z(\jr^{-2}) \p_{\leq m+1}\right) + S^Z(\jr^{-3}) \p_{\leq m+2} + f_{\le m}
\end{equation}

Due to the derivative gaining only $\ju\inv$ in the wave zone (see \cref{derbound}), we shall perform an additional decomposition as follows.
First, we note that, for any function $w$, 
\begin{equation}\label{D decomp}
\pa w \in S^Z(\jr^{-1}) w_{\leq 1} + S^Z(1) \pa_t w, \quad r\geq t/2
\end{equation}
which is clear for $\pa_t$ and $\pa_\xo$, while for $\pa_r$ we write
$\pa_r = r\inv(S- t\pat).$

Let $\chi_\text{cone}$ be a cutoff adapted to the region $t/2 \le r \le 3t/2$. We now rewrite \eqref{first write} as
\begin{equation}\label{final write}
\Box\p_J \in S^Z(\jr^{-3}) \p_{\leq m+2} + (1- \chi_\text{cone}) \left(S^Z(\jr^{-2}) \pa\p_{\leq m+1}\right) + \pa_t \left(\chi_\text{cone} S^Z(\jr^{-2}) \p_{\leq m+1}\right) + f_{\le m}
\end{equation}
	We now write $\p_J = \sum_{j=1}^3\p_j$ where the functions $\p_j$ solve
\begin{equation}\label{decomp}
\begin{split}
\Box \p_1 = G_1, \quad G_1 \in S^Z(\jr^{-3}) \p_{\leq m+2} + (1- \chi_\text{cone}) \left(S^Z(\jr^{-2}) \pa\p_{\leq m+1}\right) \\
\Box \p_2 = \pa_t G_2, \quad G_2\in \chi_\text{cone} S^Z(\jr^{-2}) \p_{\leq m+1} \\
\Box \p_3 = f_{\le m} = G_3
\end{split}
\end{equation}
We define $\p_{2}$ in order to deal with the metric terms $\hab$ near the light cone.

	Henceforth the convention in \cref{subsec:N} will apply to the symbol $n$. 
\subsection{Estimates for the fundamental solution}\label{sec:estsfdmt}

\begin{lemma}\label{conversion}
Let $m\ge0$ be an integer and suppose that $\psi : [0,\iy)\times\R^3\to\R$ solves $$\Box\psi (t,x)= G(t,x), \qquad \psi(0) = 0, \quad \pa_t \psi(0) = 0. $$ 
Define
\begin{equation}\label{def:h}
h(t,r) := \sum_{i=0}^2 \|\Omega^i G (t, r\omega)\|_{L^2(S^2)}
\end{equation}which serves as a radial majorant of $|G|$, thus $|G| \le h$. 
	Assume that %
$$ h(t,r)  \ls \f{1}{ \jr^\x \la v\ra^{\beta} \la u\ra^\eta }, \quad 2 < \x < 3, \ \ 3 < \x < \iy, \quad \beta\geq 0, \quad \eta\geq -1/2.$$
	Define%
\[
\tilde\eta = \left\{ \begin{array}{cc} \eta -2,& \eta<1   \cr -1, & \eta > 1 
  \end{array} \right. .
\]
Then inside the region $\{ u > 1\}$, we have
\begin{equation}\label{Bd1}
\psi(t, x)\ls \frac{1}{\la r\ra\la u\ra^{\alpha+\beta+\tilde\eta-1}}.
\end{equation}
\end{lemma}

\begin{proof}
The idea is to use Sobolev embedding and the positivity of the fundamental solution of $\Box$ to show that
$$r\psi \lesssim \int_{D_{tr}} \rho h(s,\rho) ds d\rho,$$
where $D_{tr}$ is the backwards light cone with vertex $(r, t)$, and use \cref{def:h}. After plugging in the decay rates for $h(t,r)$ and integrating, \cref{Bd1} is the result.
\end{proof}

For $\p_{2}$ we will use the following result for an inhomogeneity of the form $\pat G$ supported near the light cone. The result is similar to \cref{conversion}, aside from a gain of $\ju$ in the estimate in the conclusion: see \cref{Bd1der} compared to \cref{Bd1}.

For both of the problems $P\p = f$ and $P'\phi = f$ with zero initial data (recall \cref{rem:id}), we shall use \cref{lem:cone} only once, namely when we transition the bound $\p \ls \jv\inv$ into the final bound \cref{eq:main bound}. 

\begin{lemma}\label{lem:cone}
Let $\psi$ solve 
\begin{equation}\label{Mink2}
\Box \psi = \pa_t G, \qquad \psi(0) = 0, \quad \pa_t \psi(0) = 0,
\end{equation}
where $g$ is supported in $\{\f12 \leq \f{|x|}t \leq \f32 \}$ (that is, near the light cone). Let $h$ be as in \cref{def:h}, and assume that 
$$|h| + |Sh| + |\Omega h| + \la t-r\ra |\pa h| \ls \frac{1}{\la r\ra^{\alpha}\la u\ra^{\eta}},  \quad 2 < \alpha <3,  \quad \eta\geq -1/2.$$
	Then in the region $\{ u > 1\}$, when $\alpha+\eta > 3$ we have
\begin{equation}\label{Bd1der}
\psi(t, x)\lesssim \frac{1}{\la r\ra\la u\ra^{\alpha+\tilde\eta}}.
\end{equation}

\end{lemma}
\begin{proof}
Let $\tpsi$ solve $$\Box \tpsi = G, \quad \tpsi(0)= 0, \pat \tpsi(0)=0.$$
In the support of $g$, we have that the boosts of $h$ are bounded by
\[
(t \pai + x_i \pat) h \ls |Sh| + |\xO h| + \la t-r\ra |\pa_r h|
\]	where $h$ was defined in \cref{def:h}. 
	By \cref{conversion} with $\xb=0$ applied to the functions $\pa \tpsi$,  
$\xO \tpsi$, $S \tpsi$, and noting that the bound
$$\ju \partial_t \tpsi \ls |\pa \tpsi| +|S\tpsi|+|\xO \tpsi|  + \sum_i  | (t \partial_i + x_i \partial_t) \tpsi|$$ holds,  we now see that \cref{Bd1der} holds.
\end{proof}

\section{Propagating pointwise decay from $\{ r \ge t/2\}$ into $\{r\le t/2\}$}\label{sec:propag}
In this section we show how to convert the $\jr\inv$ decay for vector fields of the solution (which arises from the fundamental solution) into a factor of $\jt\inv$. Thus pointwise decay in $\{r\ge t/2\}$ is propagated into $\{ r\le t/2\}$.  
	While we intend to provide sufficient detail for a self-contained proof here, the reader looking for an alternate exposition can find it in \cite{L} or \cite{MTT}.

\begin{lemma}\label{lem:Ltwo} %
Suppose that the operator $P$ from \cref{P def} satisfies the stationary ILED \cref{def:SILED}. 
For all $0 \leq  T_1 \leq T_2$, we have
\begin{align}\label{Ltwo.est}
\begin{split}
\|\pa\pm&\|_{\lt( [T_1,T_2] \times \{ r \le t \} )}
  \ls \sum_{j=1}^2 \| \jr^{1/2} \pa \pm(T_j)\|_{\lt}  + \| \jr f_{\le m}\|_{\ltlt} + \|\pat \pm \|_{L^2L^2}. 
\end{split}
\end{align}
\end{lemma}

\begin{proof}
We demonstrate the case $m=0$ first for simplicity. We multiply the equation by $r\pa_r\phi + \phi$ and integrate by parts in $[T_1,T_2]\times\R^3$. By the assumptions on either the operator $P$ or $P'$, there exists a number $q'>0$ dependent on the decay rates of the coefficients of the operator such that
\begin{align}\label{computation}
\begin{split}
 \int |\de &\phi|^2 + O(\jr^{-q'})  |\de \phi|^2 + O(\jr^{-1-q'}) |\pao\p|^2 + O( \jr^{-2-q'}) |\phi|^2  \,dxdt \\
  &\ls \sum_{j=1}^2 \int_{\R^3} O(\la r \ra) |\de \phi(T_j,x)|^2 +  O(\la r \ra^{-1}) |\phi(T_j,x)|^2 \, dx + \int |r(P\p)\pa_r\phi| + |(P\p)\p| \,dxdt \\
  &\ls \sum_{j=1}^2 \int_{\R^3} O(\la r \ra) |\de \phi(T_j,x)|^2 \, dx + \int |r(P\p)\pa_r\phi| + |(P\p)\p| \,dxdt \\
\end{split}
\end{align}with the last statement following by a version of Hardy's inequality. 
	
	By the Cauchy-Schwarz inequality and Hardy's inequality we can bound all the terms involving $P\p$ by 
	$$\f1\eps \|rP\p\|_\ltlt^2 + \eps\|\pa_r \phi\|_\ltlt^2.$$ 	
	By using the positivity of $q'$ on the left-hand side of \cref{computation} for large $|x|$ values, we can then obtain
\begin{equation}\label{sled m=0}
\begin{split}
\| \de \phi&\|_{\lt[T_1,T_2]\lt} \ls \sum_{j=1}^2 \| \jr^{1/2} \de \phi(T_j)\|_{L^2} +  \|\jr P\p\|_{L^2[T_1,T_2]L^2} .
\end{split}
\end{equation}
This concludes the proof of the $m=0$ case. 

\item (The higher multiindex case)
We now prove \eqref{sled m=0} but for $\phi_J, J\ne \vec 0$. 
	We have
\begin{align*}
P \phi_J 	&= (P\p)_{J} + O(\jr^{-1-q'}) \de \phi_{\leq |J|} + O(\jr^{-2-q'}) \phi_{\leq |J|-1}.
\end{align*}
We multiply this by $r\pa_r \phi_J + \phi_J$. Then we integrate in $[T_1,T_2] \times \R^3$. The rest of the proof is then similar.
\end{proof}

\begin{proposition}\label{lem:aux}Assume that \cref{def:SILED} holds for $P$ \textit{and} that $\p$ solves \cref{eq:problem}. Then we have
$$\|\pm\|_{LE^1(\inte)} 
\ls T\inv \|\jr \pmn\|_{LE^1(\inte)}.$$
\end{proposition}

\begin{proof}
We may assume that $\phi$ is supported in $\inte$ because $[P,\chi_\inte]$ can be controlled where $\chi_\inte$ is a smooth cutoff that localises in space only (leaving the time variable alone).  
Thus 
$$\|\pm\|_{LE^1(C)} \ls \|\pa\pm(T)\|_{L^2_x} + \|\pa\pm(2T)\|_{L^{2}_{x}} + \|(P\phi)_{\le m}\|_{LE^*(C)} + \|\pat\pm\|_{LE(C)}, \  C:=\inte.$$
By the fundamental theorem of calculus and the Cauchy-Schwarz inequality, we have
$$\|\pa\pm(T')\|_{L^{2}_{x}} \ls T^{-1/2}\|(\pa\pm,S\pa\pm)\|_{\lt\lt}, \quad T'\in \{T,2T\}. $$
	By \cref{lem:Ltwo}, we have
\begin{equation}\label{thing}
\|\pa \pm\|_{\lt(C)} \ls \sum_{i =1}^2 \|\jr^{1/2} \pa\pm(i T)\|_{\lt} + \|\jr f_{\le m} \|_{L^2(C)} + \|\pat\pm\|_{\lt(C)}
\end{equation}
Notice 
\begin{itemize}
\item
Since $\pat = t\inv ( S - r \pa_{r} )$,
$$\|\pat\pm\|_{\lt}\ls T\inv \|(S\pm, r\pa_{r}\pm)\|_{\lt} \ls T^{-1/2} \|\jr(S\pm,\pm)\|_{\leo}.$$

\item
By the fundamental theorem of calculus and Cauchy-Schwarz,
\begin{align}\label{first1}
\begin{split}
\|\jr^{1/2} \pa\pm(i T)\|_{\lt}
    &\ls T^{-1/4} \| \jr^{1/4} \pa \pm \|_{\lt(C)} + T^{-3/4} \| \jr^{3/4} S \pa \pm\|_{\lt(C)}\\
    &\ls T^{-1/4} \| \jr^{1/4} \pa \pm \|_{\lt(C)} + T^{-1/2} \|\jr^{1/2} S\pa\pm\|_{\lt} \quad \text{since }\jr \ls T \\
    &\ls T^{-1/4} \| \jr^{1/4} \pa \pm \|_{\lt(C)} + T^{-1/2} \|\jr\pmn\|_{LE^1}
\end{split}
\end{align}

\item
we have
\begin{align}\label{first2}
\begin{split}
\|\jr (H^{\x\xb}\p \pa_{\x\xb}\p)_{\le m} \|_{\lt} 
&\ls \| \jr \pm \pa_{\x\xb}\pm \|_{\lt} \\
&\ls \| \jr\inv \pm(\jr\inv\pmn,\pa\pmn)\|_{\lt}  \quad\textrm{by \cref{2ndDeBd}} \\
&\ls \|\eps\jr^{-5/2}\pm\|_{\lt} + \|\jr\inv\pm\pa\pmn\|_{\lt} \quad \textrm{by \cref{u/v decay}}\\
&\ls \|\eps\jr^{-5/2}\pm\|_{\lt} + \|\eps \jr^{-2} \pm\|_{\lt}\\
&\ls \eps \|\pa_{r}\pm\|_{\lt} \quad \textrm{by Hardy's inequality}
\end{split}
\end{align}
We absorb this into the left-hand side of \cref{thing}. 
\end{itemize}

Thus from the estimate \cref{thing} we obtain the estimate
$$\|\pa\pm\|_{\lt}
\ls {T}^{-1/4}\|\jr^{1/4} \pa\pm\|_{L^2} + {T}^{-1/2} \|\jr \pmn\|_{LE^1}.$$
We decompose
$$\|\jr^{1/4} \pa\pm\|_{L^2} \sim \sum_{R}\|R^{1/4}\pa\pm\|_{\lt((T,2T)\times A_{R})}$$
and for $R$ sufficiently close to $T$, we bound 
$$T^{-1/4}\|R^{1/4}\pa\pm\|_{\lt((T,2T)\times A_{R})}\ls T^{-1/2} \|\jr\pmn\|_{\leo}$$
while for smaller values of $R$, we absorb to the left-hand side term $\|\pa\pm\|_{\lt}$.
\end{proof}

The next proposition uses \cref{lem:aux} to obtain better pointwise decay for $\pmn$ inside $\{ r < t/2\}$. 

\begin{theorem}\label{thm:r-t}
Let $\phi$ solve \eqref{eq:problem} with the assumption of \cref{def:SILED}. Assume that 
\begin{equation}\label{r bds}
\phi_{\le M}|_{r \le 3t/4} \ls \jr^{-1}\ju^{1/2 - q}
\end{equation}
for an $M$ that is sufficiently larger than $m$.
	Then we have 
	$$ \pm|_{r\le 3t/4}\ls  \jt\inv \ju^{1/2 - q} .$$
\end{theorem}

\begin{proof}
By \cref{derbound} and \cref{r bds},
$$ T\inv \|\jr \pmn\|_{LE^1} \ls T\inv \|\pmn\|_{LE} \ls T^{ - q}.$$	
Therefore \cref{lem:aux} implies
$$\|\pm\|_{LE^1(\inte)} \ls T^{- q}$$
and the conclusion now follows by \cref{DyadLclsd}. 
\end{proof}

\section{The pointwise decay iteration} \label{sec:iter}

\begin{theorem}\label{thm:iteration}
If $u>1$, then 
\begin{equation}
\label{full}\pm \ls \jv\inv \ju^{-\min(\xd,1)}.
\end{equation}
\end{theorem}

\begin{proof} %
Suppose first that $\xd < 1$. We begin by showing that $\pmn \ls \jr\inv$. By \cref{2ndDeBd}, we have
$$(\Hab\p\paab\p)_{\le m+n} \ls (\f\ju\jt)^{\xk} (\f{\ju^{1/2}}\jt)^{2} ( \f\jt{\jr\ju})^{2}
	= (\f1\jt)^{\xk} \ju^{-1+\xd} \jr^{-2}. $$
Thus in the notation of \cref{conversion}, we have $\x=2, \xb = \xd, \eta = 1 - \xd$. We can change this to $\x = 2+, \beta = \xd -, \eta = 1 - \xd$ because we have $\jr \ls \jt$ here (note that $\xd>0$ is positive). 
By \cref{Bd1} we conclude $\pmn \ls \jr\inv$. By \cref{thm:r-t}, $\pmn \ls \jv\inv$. 

In the notation of \cref{conversion}, we now have $\x +\xb = 2 + \xd$ and $\eta = 2-\xd$.
Thus we have $\jr\pmn\ls \ju^{-\xd}$. By \cref{thm:r-t}, we obtain \cref{full} for $\xd$ values that are smaller than 1, i.e. we conclude the proof of the main theorem in the small-$\xd$ case. 

For the large-$\xd$ ($\delta \ge 1$) case, note 
\begin{align}\label{tang.prop}
\begin{split}
(\pat+\pa_{r}) \p \in S^{Z}(\frac{ u}{t}) \pa \p + \frac{1}{t} S\p, \quad 
\pao \p \in S^{Z}(\la r\ra\inv) \xO\p
\end{split}.
 \end{align}
Recall that 
$$H^{\x\xb}\pa_{\x\xb}\p \in H^{\Lbar\Lbar} \Lbar^2\p + S^{Z}(1) \bar\pa \pa\p, \quad \Lbar := \pat - \pa_{r}.$$
Utilising \cref{tang.prop}, we see that near the light cone, the term $S^{Z}(1)\bar\pa\pa\p$ gains $\f\ju\jt$, so that any extra factor of $(\f\ju\jt)^{\xd-1}$, in the $\xd \ge1$ case, gains nothing.  More precisely, by \cref{met.ass,tang.prop}, 
\begin{align*}
|(H^{\x\xb}\pa_{\x\xb} \p)_{\le m}| 
&\ls \f{\ju^\xd}{\jt^\xd}|\pa^2\pm| +  |\bar\pa\pa\pm|\\
&\ls \f{\ju^\xd}{\jt^\xd}|\pa^2\pm| +  \f\ju\jt |\pa^{2}\pm| + \f1\jt|\pa\pm|\\
&\ls (\f\ju\jt)^{\min(\xd,1)}|\pa^{2}\pm| + \f1\jt|\pa\pm|\\
&\ls (\f\ju\jt)^{\min(\xd,1)}\big(\mu\inv |\pa\pmn| + \mu\inv\jr\inv|\pmn| \big).\end{align*}
We now proceed in the same way as earlier in this proof to obtain \cref{full} for this large-$\xd$ case. This completes the proof of \cref{thm:main} because the decay rate \cref{u/v decay} suffices in the $\{u < 1\}$ region (recall the remark made in \cref{outline of iteration}). 
\end{proof}

\begin{remark}
For part (2) of \cref{thm:main}, the iteration also uses \cref{conversion} and can be found in \cite{L} or \cite{Loo22.II}. The linear problem $P'\p=0$ with zero initial data reaches a decay rate $\pm\ls \jv\inv\ju^{-1-}$ and so in particular the final decay rate for $P'\phi = f(\pa^2\p,\p,t,x)$ with zero initial data is still given by \cref{full}. 
\end{remark}

\section{Appendix: the two metrics $g_B$, transforming ILED-type estimates under the coordinate normalisation, and explaining the assumption \cref{met.ass}}\label{app}

\subsection*{More on the Schwarzschild metric}

The apparent singularity at the surface $r=2M$ (called the event horizon) is merely a coordinate singularity: defining the Regge-Wheeler tortoise coordinate
$$r^*:= r + 2M\log (r-2M) - 3M - 2M\log M$$
and setting $v = \ti t + r^*$, we express the metric in the $(r,v,\xo)$ coordinates is 
$$ds^2 = -\left( 1 - \f{2M}r \right) dv^2 + 2dvdr + r^2d\xo^2$$
and this extends analytically into the black hole region $r<2M$.

\newcommand{\gK}{g^{\text{Kerr}}}

\subsection*{The Kerr metric}  Let $M>0$ denote the mass of the black hole and let $aM$ denote its angular momentum, thus $a$ denotes the angular momentum per unit mass. The Kerr metric with mass $M$ and angular momentum $aM$ is the Ricci flat metric %
given in Boyer-Lindquist coordinates $(t,r,\theta,\varphi)$ \footnote{which are a generalisation of Schwarzschild coordinates. The black hole rotates in the $\varphi$ direction.
} by
$$ds^2 = g^K_{tt}dt^2 + g_{t\varphi}dtd\varphi + g^K_{rr}dr^2 + g^K_{\varphi\varphi}d\varphi^2
 + g^K_{\theta\theta}d\theta^2$$
 where $t \in \R$, $r > 0$, $(\varphi,\theta)$ are the spherical coordinates on $\S^2$ and
\[
 g^K_{tt}=-\frac{\Delta-a^2\sin^2\theta}{\rho^2}, \qquad
 g^K_{t\varphi}=-2a\frac{2Mr\sin^2\theta}{\rho^2}, \qquad
 g^K_{rr}=\frac{\rho^2}{\Delta},
 \]
\[ g^K_{\varphi\varphi}=\frac{(r^2+a^2)^2-a^2\Delta
\sin^2\theta}{\rho^2}\sin^2\theta, \qquad g^K_{\theta\theta}={\rho^2},
\]
with
\[
\Delta=r(r - 2M) +a^2, \qquad \rho^2=r^2+a^2\cos^2\theta.
\]
The inverse of the metric is then:
\[ g_K^{tt}=-\frac{(r^2+a^2)^2-a^2\Delta\sin^2\theta}{\rho^2\Delta},
\qquad g_K^{t\varphi}=-a\frac{2Mr}{\rho^2\Delta}, \qquad
g_K^{rr}=\frac{\Delta}{\rho^2},
\]
\[ g_K^{\varphi\varphi}=\frac{\Delta-a^2\sin^2\theta}{\rho^2\Delta\sin^2\theta}
, \qquad g_K^{\theta\theta}=\frac{1}{\rho^2}.
\]
We remark that the coefficients are completely explicit. 
	Observe that the coefficients are stationary and axially symmetric (that is, axisymmetric). That is, they are independent of $t$ and $\varphi$. The Killing vector fields associated with these two symmetries are $\f{\pa}{\pa t}$ and $\f{\pa}{\pa \varphi}$ respectively.

The reader can verify that substituting $a=0$ yields the Schwarzschild metric. 
In this article we assume that $0 < a \ll M$, so that the Kerr
metric is a small perturbation of the Schwarzschild metric. 

In the introduction, we mentioned that in the Kerr case, $r_e$ is a number such that $r_e \in (r_-,r_+)$ where $r_->0$ corresponds to the Cauchy horizon and $r_+$ to the event horizon. Viewing $\Delta$ as a quadratic polynomial in $r$, one can then see that the constants $r_\pminus$ are %
the two roots of $\Delta(r)$.

\subsection*{A remark on the energy norm and ILED norm in the normalised coordinates}
While \cref{EB,eq:SILED} are stated in terms of the original coordinates, the conclusion of the main theorem is, on the other hand, stated in terms of normalised coordinates.  We note here that \cref{EB,eq:SILED} are changed only minimally when transitioning from the original coordinates to the normalised coordinates. The change of coordinates only affects $\{ r\gg1\}$. %
The initial surface is changed by an $O(\log r)$ amount (see \cite{Tat}), but this is not an issue. 
Changing \cref{EB} can be done by standard energy estimates for the operator $\Box_g$.

\subsection*{Additional explanation for the assumption on $H^{\Lbar\Lbar}$ in \cref{met.ass}} Here we provide some context to the assumption on $H^{\Lbar\Lbar}$ in \cref{met.ass}. We shall express $\Hab\p + O(\p^2)$ in a null frame in the metric $g_B$:
$$\pa_{r^*} = (1 - \f{2M}r) \xo^i\pai, \quad \Lbar = \pat - \pa_{r^*}, \quad L =  \pat + \pa_{r^*}, \quad A = A^i(\xo)\pai, \quad A_* = A_*^i(\xo)\pai$$
where $L$ and $\Lbar$ are null vectors and $A,A_*$ are orthonormal vectors
$$g_B(L,L) = g_B(\Lbar,\Lbar) = 0, \qquad g_B(L, \Lbar) = -2(1 - \f{2M}r)$$
$$g_B(A,A) = g (A_*,A_*) = 1, \qquad g_B(A, A_*) = 0.$$
In addition, these four vectors $L, \Lbar, A, A_*$ are tangential to the constant-$r$ spheres:
$$g_B(L,\calV) = 0 = g_B(\Lbar, \calV) \qquad \calV \in \{A, A_*\}.$$
In what follows, we explain how the additional decay assumption on $H^{\Lbar\Lbar}$ in \cref{met.ass} arises. We expand $\hab := \Hab\p + O(\p^2)$ in our null frame. Let $\calT := \{ L,A,A_*\}$, with $\calT$ for ``tangential.''
\begin{align}
\begin{split}
h^{\alpha\beta}
&=h^{\uL\uL}{\uL}^\alpha{\uL}^\beta
+ \sum_{T\in \mathcal{T}} h^{\uL T} ({\uL}^\alpha T^\beta +T^\alpha {\uL}^\beta)
+\sum_{U,T\in \mathcal{T}} h^{UT} U^\alpha T^\beta
\end{split}
\end{align}and this indicates that the first term in the sum decays slowest, because the $\Lbar$ derivative of the solution $\p$ decays more slowly than all the other derivatives of $\p$ near the light cone. Thus we assume faster pointwise decay on $h^{\uL\uL}$ near the light cone, and this amounts to assuming faster pointwise decay on $H^{\uL\uL}$ near the light cone, as seen in \cref{met.ass}.

\end{document}